\documentclass[11pt,centertags]{amsart}
\usepackage{amssymb}
\usepackage{enumitem}
\usepackage{url}
\usepackage{gloss}
\usepackage[pdftex]{graphicx}
    \DeclareGraphicsExtensions{.eps,.pdf,.png,.jpg}
    \graphicspath{{}}
\usepackage{mathtools}
\usepackage[marginratio=1:1,height=8.5in,width=6.5in,tmargin=1.25in]{geometry}
\usepackage{comment}
\usepackage[normalem]{ulem}
\usepackage[usenames,dvipsnames]{color}
\usepackage{todonotes}

\usepackage[utf8]{inputenc}
\usepackage{textcomp}
\usepackage{bbm}
\usepackage{mathpazo}
\usepackage{mathrsfs}

\newtheorem*{main*}{Main Theorem}

\newtheorem{theorem}{Theorem}[section]
\newtheorem*{theorem*}{Theorem}
\newtheorem{proposition}[theorem]{Proposition}
\newtheorem{lemma}[theorem]{Lemma}
\newtheorem{corollary}[theorem]{Corollary}

\newtheorem*{question*}{Question}

\newtheorem*{conjecture*}{Conjecture}

\newtheorem*{Assumption}{Assumptions on the Matrix Ensemble}

\theoremstyle{definition}

\newtheorem{definition}[theorem]{Definition}
\newtheorem*{definition*}{Definition}
\newtheorem{example}[theorem]{Example}

\theoremstyle{remark}
\newtheorem{remark}[theorem]{Remark}

\numberwithin{equation}{section}

\newcommand{\mc}[1]{{\mathcal #1}}
\newcommand{\mbf}{\mathbf}
\newcommand{\mbb}{\mathbb}

\newcommand{\cout}[1]{}
\definecolor{darkcyan}{rgb}{0. 0.65, 0.65}

\newcommand{\E}{\mathbb{E}}
\newcommand{\C}{\mathbb{C}}

\newcommand{\eps}{\epsilon }

\DeclareMathOperator{\tr}{tr}

\DeclareMathOperator{\dist}{dist}

\DeclareMathOperator{\vol}{Vol}
\DeclareMathOperator{\Vol}{Vol}

\DeclareMathOperator{\proj}{proj}
\DeclareMathOperator{\Gr}{Gr}
\DeclareMathOperator{\Graff}{Graff}
\DeclareMathOperator{\Av}{Av}
\DeclareMathOperator{\Mom}{Mom}

\newcommand{\op}[1]{\operatorname{#1}}
\newcommand{\set}[1]{\left\{#1 \right\}}

\providecommand{\to}{\longrightarrow }
\newcommand{\tensor}{\otimes }

\newcommand{\abs}[1]{\left\lvert #1 \right\rvert }
\newcommand{\norm}[1]{\left\| #1 \right\| }

\def\be#1\ee{\begin{align}\begin{split} #1 \end{split}\end{align}}
\def\beq#1\eeq{\begin{align*}\begin{split} #1 \end{split}\end{align*}}

\renewcommand{\hat}{\widehat}
\renewcommand{\bar}{\overline}
\renewcommand{\(}{\left(}
\renewcommand{\)}{\right)}

\renewcommand{\P}{\mathbb{P}}
\renewcommand{\S}{\mathbb{S}}

\begin{document}

\title{The Circular Law for Random Matrices with Intra-row Dependence}

\author[Chris Connell]{Chris Connell$^\dagger$}
\thanks{$\dagger$ Supported in part by Simons Foundation grant \#210442}
\author[Pawan Patel]{Pawan Patel}

\address{Indiana University}
\email{cconnell@indiana.edu}

\address{Indiana University}
\email{pawpatel@indiana.edu}


\begin{abstract}
We consider the problem of determining the limiting spectral distribution for random matrices whose row distributions are permitted to have limited dependence. We assume mild moment conditions and give an extension of the Mar\v{c}enko-Pastur theorem for this context. The main new feature here are geometric conditions on the distributions which allow us to extend the circular law to this setting.
\end{abstract}

\maketitle

\thispagestyle{empty}

\section{Introduction} 

A central point of interest in the theory of random matrices is spectral universality, i.e. the degree to which the eigenvalues of a matrix randomly chosen from a given ensemble will follow a particular density law independent of the choice of matrix and sometimes, within certain limits, on choices that govern the ensemble such as the distributions for the entries. For a broad survey from a historical perspective, see \cite{Diaconis17}\cout{ or \cite{Kopel16}}. A key example of universality comes from the Tracy-Widom distribution \cite{Tracy98} and the far reaching results of Tao and Vu \cite{Tao08,Tao09b} on the circular law. Like most of the results obtained until recently, these have focused on the the classical setting where one assumes independence of the entries.

More recently, a number of authors have attacked various generalizations and analogues (see for example, \cite{Bai93,Benaych-Georges12,Bleher16,Bloemendal16,Chakraborty16,Costello06,Dumitriu16,Edelman16,Erdoes10,Gamboa17,Halvorsen16,Ke16,Liu16a,Loubaton16,Male17,Tao10,Tao12,Tao14a,Tao15,Yao12}). Among these generalizations a number of recent results have begun to explore universality under the allowance for a (necessarily) limited amount of dependence between the entries (e.g. \cite{Adamczak08,Adamczak11,Adamczak15,Liu15,Wood16}). Among the latter category include the original Mar\v{c}enko-Pastur paper (\cite{Marcenko67}), where some dependence within rows was allowed, but for a spherically uniform distribution. This was generalized by Pajor and Pastor (\cite{Pajor07}) to allow for an arbitrary isotropic log-concave distribution.

In this paper we are concerned with exploring the limits to which dependence can be allowed in the current best approach to achieving the circular law. 
As in several recent results, we use the Tao-Vu replacement principle (\cite{Tao08}) along with a generalization of the Mar\v{c}enko-Pastur Law to our situation to handle the middling and large eigenvalues. Combined with our geometric conditions on the row distributions, we are able to obtain the appropriate bounds on the lowest singular values of the random matrices to obtain the circular distribution in our dependent case.

We need three main assumptions on the distribution of the entries of our random matrices in order to obtain the required Mar\v{c}enko-Pastur law. These are only used for this purpose, and without these assumptions our results would still follow if the Mar\v{c}enko-Pastur component can be guaranteed by other means.

For an ensemble of $n\times N_n$ random matrices $A_n$, denote by $\Mom_{k}(A_n)$ the expected value of the $k$-fold Kronecker (tensor) product $A_n\tensor\cdots\tensor A_n$.

\begin{Assumption}
Let $(N_n)$ be an increasing sequence of positive integers with $\lim_{n\to\infty} \frac{n}{N_n}\in (0,\infty)$. We assume that the distributions on our ensembles of $n\times N_n$ random matrices $A_n$ with entries $x_{ij}^{(n)}$ satisfy:
\begin{enumerate}

\item[(A1)] for every $k\in \mathbb{N}$, $\sup_n \max_{1\leq i\leq n,1\leq j\leq N_n} \E[|x_{ij}^{(n)}|^k] < \infty$;

\item[(A2)] for every $k\in \mathbb{N}$, the sum of all terms in $\Mom_{2k}(A_n)$ with at least one $x_{ij}^{(n)}$ appearing with a power of 1 is of size $o_k(n^{k+1})$;

\item[(A3)] for every $\epsilon > 0$:
$$ \lim_{n\rightarrow \infty} \frac{1}{n} \sum_{i \leq n} 
\P\( \abs{ \frac{1}{N_n}\sum_{j=1}^{N_n} (x_{ij}^{(n)})^2 - 1 } \geq \epsilon \) = 0 $$

and 

$$ \lim_{n\rightarrow \infty} \frac{1}{N_n} \sum_{j \leq N_n} 
\P\(\abs{\frac{1}{n}\sum_{i=1}^{n} (x_{ij}^{(n)})^2 - 1 } \geq \epsilon \) = 0.$$
\end{enumerate}
\end{Assumption}

The first and third of these conditions appear in Adamczak. The third assumption requires that the rows (resp. columns) of the random matrix $A_n$ have Euclidean norm, normalized by $\sqrt{N_n}$ (resp. $\sqrt{n}$) converge in probability to 1. This is necessary if one expects to have a universality result of this kind.\\

The main difference between our version and earlier results is assumption A2, which allows for more dependence in that it only requires a particular sum of the $2k$-th moment of the random vector to decay sufficiently quickly. Previous results used a stronger assumption, namely that for every $n, i, j$ the $\E(x_{ij}^{(n)} | \mathcal{F}_{ij} )  = 0$ where $\mathcal{F}_{ij}$ is the $\sigma$-field generated by $\{x_{kl}^{(n)} : (k,l) \neq (i,j) \}$. \\

Our main result is the following. (For the definition of $\norm{\nu_1^{(n)},\dots,\nu_n^{(n)}}_{d, \delta}$, see Section \ref{sec:singular}.)

\begin{theorem}
Let $A_n$ be a sequence of $n \times n$ random matrices with independent rows $X_1^{(n)}, ..., X_n^{(n)}$ defined on a common probability space and satisfying assumptions A1-A3. Assume that for each $n$ and $i,d\leq n$, and all $\delta>0$, the probability measures $\nu_i^{(n)}$ for $X_i^{(n)}$ have uniformly bounded $\norm{\nu_1^{(n)},\dots,\nu_n^{(n)}}_{d, \delta}$. Then almost surely the spectral measure of $\mu_{\frac{1}{\sqrt{n}}A_n}$ converges weakly to the uniform distribution on the unit disk in $\C$. 
\end{theorem}

\begin{remark}
The hypotheses are quite close to necessary in a certain sense, although there still seems to be room to slightly weaken the hypotheses A1-A3.

Also, one only needs $\delta\geq n^{-\frac52+\beta}$ for almost sure convergence and $\delta\geq n^{-\frac32+\beta}$ for convergence in probability. (Here $\beta>0$ is any small number.) Moreover the bound on the $\norm{\cdot}_{d,\delta}$ is only needed for $d<n-n^{0.99}$, but this constraint is not very restrictive to begin with (see Proposition \ref{prop:bounded}).
\end{remark}

\section{Least Singular Value}\label{sec:singular}

While our results are stated in terms of arbitrary (Radon) probability measures for the distributions of the rows of our random matrices, it is convenient to work with absolutely continuous measures on $\C^n$. The general case is recovered by passing to weak-* limits. 

For $f \in L_{loc}^1(\C^m)$ and any bounded Borel set $E \subset \C^m$. We define $Av_{E}(f) = \frac{1}{vol_m(E)} \int_E f(x)dx$. If not explicitly specified, then the dimension $m$ in the definition will be the minimal one for which $E$ belongs to an $m$ dimensional affine subspace, in case it belongs to a larger copy of $\C^n$. \\

Let $\omega_{n} = \frac{\pi^{n}}{\Gamma(n+1)}$ denote the volume of the unit ball in $\C^n$. The volume of the sphere $\S^{2n-1}(r)$ of dimension $2n-1$ and radius 
$r$ in $\C^n$ is then $2n\omega_n r^{2n-1}=\frac{2r^{2n-1}\pi^{n}}{\Gamma(n)}$. Moreover, 
let $r_n=(2n\omega_{n})^{\frac{-1}{2n-1}}$ so that $\Vol(\S^{2n-1}(r_n))=1$. More generally, let 
$\S_A(r)$, resp. $B_A(r)$, denote the sphere, resp. ball, of radius $r$ around $0$ in the 
subspace $A$. For any set $E\subset \C^n$ we denote by $x+E$ the translation of $E$ by $x\in \C^n$.  \\

Finally, let
$\Gr_{n,d}$ represent the Grassmanian of all $d$-dimensional linear subspaces of $\C^n$. We define the following norms.

\definition
For any $W\in \Gr_{n,n-d}$ with orthogonal subspace $W^\perp\in \Gr_{n,d}$ and any Borel function $f: \C^n \rightarrow [0, \infty)$, we set
\begin{align*}
\norm{f}_{W,\delta,1}&=\int_{W} \Av_{x+B_{W^\perp}(\delta)}(\abs{f}) dx \\
\end{align*}
and 
\[
\norm{f}_{d, \delta, 1} = \sup_{W \in \Gr_{n,n-d}} \norm{f}_{W, \delta, 1}
\]

\definition
\[
\norm{f}_{d,\delta,2}=(n-d)\omega_{n-d}\int_0^\infty 
\norm{f}_{L^\infty\left(\bigcup_{t'\in[t,t+\delta]}\S^{2n-1}(t')\right)} t^{n-d-1}  dt
\]\\

Given the that $f$ is a probability measure on $\C^n$, Def. 2.1 is the probability that the projection of a random vector drawn from $f$ to a fixed $n-d$ dimensional subspace, $W$, will be less than $\delta$. This quantity is related to the probability that a singular value for a random matrix with independent rows, distributed according to $f$, is small. Taking the supremum over $W \in \Gr_{n, d}$ yields, in some sense, the worst subspaces $W$ where a projection is likely to be small causing a singular value for the random matrix to also be small. Thus, this is a quantity one would like to control. Def 2.2 allows one to express Def. 2.1 without reference to any subspace, without too much loss of precision. It yields a condition that is easier to check in that the role of $W$ disappears. Moreover, it still allows for $f$ to have a pole of order at most $d-1$ at the origin. We have the following two lemmas: \\ 

\begin{lemma}\label{lem:ineqs}
$$\lVert f \rVert _{d, \delta, 1} \leq \lVert f \rVert _{d, \delta, 2} $$ \\
and 
$$ \norm{f}_{W, \delta, 1} \leq \frac{\norm{f}_1}{\omega_d \delta^d} $$
\end{lemma}

\begin{proof} 
The second inequality follows from the definition of the average, Fubini and the fact that $\vol(B_{W^\perp}(\delta))=\omega_d\delta^d$.

To get the first inequality, note that 
\begin{align*}
\norm{f}_{W,\delta,1}&=\int_{W} \Av_{x+B_{W^\perp}(\delta)}(f) dx \\ 
&=(n-d)\omega_{n-d} \int_0^\infty \Av_{\S_{W}(t)\times B_{W^\perp}(\delta)}(f) t^{n-d-1}  dt \\
&\leq \big\lVert \lVert f \rVert_{L^{\infty}(y+B_{W^\perp}(\delta))} \big\rVert_{L^1(W)}\\
&\leq (n-d)\omega_{n-d}\int_0^\infty \norm{f}_{L^\infty(\S_{W}(t)\times B_{W^\perp}(\delta))} t^{n-d-1}  dt
\end{align*}

For
any $t>0$, and $W\in \Gr_{n,n-d}$,  
$$\S_{W}(t)\times B_{W^\perp}(\delta)\subset \cup_{t'\in
    [t,\sqrt{t^2+\delta^2}]} \S^{n-1}(t')$$ 
    Hence the essential sup of $f$ on
$\S_{W^\perp}(t)\times B_{W^\perp}(\delta)$ will be achieved on $\S^{n-1}(t')$
for some $t'\in[t,\sqrt{t^2+\delta^2}]\subset [t,t+\delta]$.
\end{proof}

\begin{lemma}\label{lem:proj1} (Small Eigenvalue Lemma)
Let $f_X$ be the PDF for the random variable $X$ taking values in $\C^n$. For any $W\in 
\Gr_{n,n-d}$ we have
\[
\P(\norm{\proj_W(X)}\leq \delta)\leq \omega_d 
\norm{f_X}_{W,\delta,1} \delta^d.
\]
\end{lemma}

\begin{proof}
For fixed $n-d$-dimensional subspace $W\subset \C^n$ we let $$H_W(\delta)=\set{x\in \C^n: 
\norm{\proj_{W^\perp}(x)}\leq \delta}$$ 

Note that $H_W(\delta)$ is just the $\delta$-neighborhood of $W$.
So we may write $H_W(\delta)=W \times B_{W^\perp}(\delta)$. 

Hence we may write,
\[
\begin{split}
\P(\norm{\proj_{W^\perp}(X)}\leq \delta)&= \int_{H_W(\delta)} f_X(x) dx \\
&=\int_{W}\int_{B_{W^\perp}(\delta)} f_X(y+z)dz dy\\
&=\omega_d \delta^d  \int_{W} \Av_{y+B_{W^\perp}(\delta)}(f) dy.
\end{split}
\]
\end{proof}

Taking supremums over $W$ in the previous lemma, we have the following.
\begin{corollary}
    If $X$ has probability distribution $f_X$ on $\C^n$ then 
    \[
    \P(\inf_{W\in \Gr_{n,n-d}} \norm{\proj_{W^\perp}(X)}\leq \delta)\leq \omega_d 
    \norm{f_X}_{d,\delta,1}\delta^d.
    \]
\end{corollary}
\vspace{2pc}

Note that since $\norm{f_X}_1=1$ we have from Lemma \ref{lem:ineqs} that $\norm{f_X}_{d,\delta,1}\leq \frac{1}{\omega_d \delta^d}$. However, we want to exploit the $\delta$-decay in Lemma \ref{lem:proj} below, so we would like a bound $\norm{f_X}_{d,\delta,1}\leq C$ which is independent of $n$ or $\delta$.

We first describe some examples where this doesn't happen. That is, where 
$\norm{f_X}_{1,\delta,1}\geq \frac{C}{\delta^d \omega_d}$, and is therefore unsuitable for the 
estimates we need.

\begin{example}
Suppose that $X$ has iid entries each of which is a Bernoulli variable $f_i$ with point masses 
at $\pm 1$. Then consider any choice of codimension d plane $W^\perp$ which passes through the 
origin and through $\frac{1}{2^d}$ of the vertices, ${1,-1}^n$, of the $n$-cube of sidelength 
$2$. Such a plane can be chosen as the span of any $n-d$ independent vectors with entries $\pm 
1$. In this case, $\norm{f_X}_{1,\delta,1}\geq \norm{f_X}_{W,\delta,1}=\frac{1}{2^d\omega_d 
\delta^d}$. 
\end{example}

A more obvious problem is the following.

\begin{example}
Consider a random vector $X\in \C^n$ whose PDF $f_X$ is concentrated completely in an $\eps$ 
neighborhood of $W^\perp=\C^{n-d}$, i.e. it's support is $\C^{n-d}\times B_{\C^d}(\eps)$. We could even have $f_X$ be bounded, and then 
\[
\norm{f_X}_{d,\delta,1}\geq 
\norm{f_X}_{W,\delta,1}=\frac{\frac{\delta^d}{\eps^d}\norm{f_X}_1}{\omega_d\delta^d}=\frac{1}{\omega_d\eps^d}.
\]
In particular, this is bounded but arbitrarily badly as $\eps\to 0$.
\end{example}

We can generalize this last example considerably to obtain differing behaviors.
\begin{example}
Now consider a fixed subspace $W$ of dimension $d$ and $f_X:W^\perp\times W\to [0,\infty)$ of the form $f_X(x,y)=C\chi_{x+B_W(g(\norm{y}))}$ for some function $g:(0,\infty)\to (0,\infty)$.
Employing polar coordinates, the condition that $\norm{f_X}_1=1$ becomes the condition 
\[
C=\left((n-d)\omega_{n-d}\omega_d \int_0^{\infty} r^{n-d-1} g(r)^d dr \right)^{-1}
\]
and we may evaluate
\[
\norm{f_X}_{W,\delta,1}=\frac{1}{\omega_d\delta^d}\left[ 1-C(n-d)\omega_{n-d}\omega_d \int_{0}^{g^{-1}(\delta)} r^{n-d-1} (g(r)^d-\delta^d) dr \right].
\]
Now we can specialize to the case
\[
g(r)^d=r^{1+d-n}\begin{cases} 1 & r\leq 1 \\ r^{-1-\alpha} & r>1 \end{cases}
\]
for some choice of $\alpha\in(0,\infty)$.
We may explicitly compute $\frac{1}{C}=(n-d)\omega_{n-d}\omega_d(1+\frac{1}{\alpha})$ and for $\delta<1$, $g^{-1}(\delta)=\delta^{\frac{-d}{n+\alpha-d}}$. Another explicit computation gives
\[
\omega_d\delta^d\norm{f_X}_{W,\delta,1}=\frac{\frac1{\alpha}+\frac1{n-d}}{\frac1{\alpha}+1} \delta^{\frac{d\alpha}{n+\alpha-d}}.
\]
As we will need to take $\delta=n^{-\frac52-\beta}$ for some small $\beta>0$ (see Theorem \ref{thm:leastsing}), if we take $\alpha$ sufficiently close to $0$ then the above decays slower than $n^{-\frac52}$ and our necessary estimates fail. On the other hand for sufficiently large $\alpha$ the right hand side approaches $\frac{\delta^d}{n-d}$ and this is eventually smaller than $\frac{n^{-\frac52}}{n-d}=O(n^{-\frac72})$, and the required estimates succeed.
\end{example}

We now describe some general cases where $\norm{f}_{d,\delta,1}$ is bounded independently of $\delta$. For $0\leq d\leq n$, let $\Graff_{n,d}$ be the affine Grassmanian of all $d$-dimensional affine spaces of $\C^n$. Define the generalized Radon transform of $f\in L^1(\C^n)$ as the measurable function $\mc{R}_d(f):\Graff_{n,d}\to\hat{\C}$ given by
\[
\mc{R}_d(f)(W)=\int_{W}f.
\]
(Here the measure is the Lebesgue measure and we must both allow for infinite values and accept that the Radon Transform is not invertible on all of $L^1(\C^n)$.)
\begin{proposition}\label{prop:bounded}
If $f_X$ satisfies any of the following conditions,
\begin{enumerate}
\item The row $X$ consists of independent entries with PDF $f_i$ (not necessarily i.d.) and 
$\norm{f_i}_\infty< \frac{C^{\frac1d}}{\sqrt{2}}$ 
\item $\norm{t^{n-d-1}\sup_{\S^{n-1}(t)}{f_X}}_{L^1((0,\infty))}<\frac{C}{(n-d)\omega_{n-d}}$
\item $\norm{\mc{R}_{n-d}(f_X)}_\infty \leq C$
\end{enumerate}
Then $\norm{f_X}_{d,\delta,1} < C$.
\end{proposition}

\begin{proof}
For the first condition note that $f_X(x)=\prod f_i(x_i)$, and suppose each is bounded by 
$A=\frac{C^{\frac1d}}{\sqrt{2}}$. For any choice of $W^\perp$, and 
$y\in B_{W^\perp}(\delta)$, there is a projection $P$ onto one of the 
$\binom{n}{n-d}$ choices of $n-d$-coordinate planes, say the first $n-d$ coordinates, so 
that for all $w\in W$, $d(P(w),w)\leq \norm{P(w)}$. In other words, $y+W$ is the graph of a linear map $L:\C^{n-d}\to \C^d$ followed by a translation where $\norm{L}_{\rm op}\leq 1$. The volume distorsion of the corresponding graph map $(I,L):\C^{n-d}\to \C^n$ is then 
\[
\sqrt{\det({I+L^*L})}\leq \sqrt{2}^d.
\]
We may then write for any $y\in B_{W^\perp}(\delta)$,
\begin{align*}
\int_{W} f_X(y+w) dw&=\int_{W} \prod_{i=1}^n f_i(w_i) 
dw\\
&\leq \int_{R^{n-d}} \prod_{i=1}^{n-d} f_i(x_i)\prod_{i=n-d+1}^n 
f_i((Lx)_i)  \sqrt{\det(I+L^*L)} dw \\
&\leq \int_{R^{n-d}} \prod_{i=1}^{n-d} f_i(x_i) A^d  \sqrt{2}^d dw \\
&\leq \sqrt{2}^d A^d=C.
\end{align*}
Here we have used that the $f_i$ are individually PDF's.
Finally, taking the average over $y\in B_{W^\perp}(\delta)$ does not change this.

The second statement amounts to $\norm{f_X}_{d,\delta,2}<C$, and the statement follows by Lemma \ref{lem:ineqs}. The last condition states that for any $W\in \Graff_{n,n-d}$, every translate of $W$ has $f_X$ integral bounded by $C$. Hence 
\[\norm{f_X}_{d,\delta,1}=\frac{1}{\omega_d\delta^d}\int_{B_{W^\perp}(\delta)} \int_{W} f_X(x+y)dxdy\leq \frac{1}{\omega_d\delta^d}\int_{B_{W^\perp}(\delta)} C dy=C.\]
\end{proof}

It will turn out that Proposition 2.8 gives sufficient criteria for the singular values of a random matrix whose rows are independently drawn from $f$ to be small enough for the Circular Law to possibly hold. However, Example 2.6 demonstrates that when all of these conditions fails the Circular Law may still hold. Indeed, in the case of a random matrix of i.i.d. Bernoulli entries, Tao and Vu (\cite{Tao07}) have showed the circular law holds. On the other hand, there is no bound of the form (1),(2) or (3) from Proposition 2.8 for a Bernoulli random vector, even with independent entries. 

The issue at hand, of course, is that the Bernoulli random vector has a support measure with atoms that have unbounded Lebesgue integrals on lower dimensional slices. By taking the supremum over $W$ in in Definition 2.1 to arrive at Definition 2.2, we select the worst case $W$ for our purposes and, in the case of unbounded $f$, are doomed. To deal with this, we may instead take the expectations over $W$.

\begin{definition}
\[ 
\norm{f}_{d,\delta}=\E\left[\norm{f}_{W,\delta,1} \right]
\]
where the expectation is taken over $W$ spanned by $n-d$ vectors $X_1,\dots, X_{n-d}$ with the corresponding joint expectation induced from $f_{X_i}=f$. (We will usually assume that the vectors are chosen independently so that the joint distribution simplifies.)

More generally if we have $n$ vectors $X_1,\dots,X_n$ in $\C^n$ randomly chosen via corresponding independent distributions $f_1,\dots,f_n$ then we define
\[ 
\norm{f_1,\dots,f_n}_{d,\delta}=\sup\E\left[\norm{f_i}_{W,\delta,1}\right]
\]
Where the sup is over all $i$ and all measures $\mu_{i,k}$ and $\mu_{i,l}'$ on $G_{n,n-d}$. Here $\mu_{i,k}$ for $1\leq k\leq \binom{n-1}{n-d}$ is the measure induced on $\Gr_{n,n-d}$ from the distributions $f_{i_1},\dots,f_{i_{n-d}}$ with $i\not\in\set{i_1,\dots,i_{n-d}}$ and $\mu_{i,l}'$ for $1\leq l\leq \binom{n-1}{d}$ is the pushforward of $\mu_{i,l}$ on $\Gr_{n,d}$ under the map $W\mapsto W^\perp$. More specifically, for $E\subset \Gr_{n,n-d}$, $\mu_{i,k}(E)$ is the probability that the $n-d$ vectors with distributions $f_{i_1}$,\dots, $f_{i_{n-d}}$ span a subspace in $E$ where the choice of indices come from the $k-th$ permutation. (Note that the probability that the span is lower dimensional is zero.) In what follows, a ``random subspace'' will mean one chosen with respect to one of these distributions.
\end{definition}

\begin{remark}
Recall, that the map $V\mapsto V^{\perp}$ induces an isometry between $\Gr_{n,d}$ and $\Gr_{n,n-d}$. However this map does not necessarily push forward the measure $\mu_{i,l}$ on $\Gr_{n,d}$ to any of the $\mu_{i,k}$ on $\Gr_{n,n-d}$. This is why we had to use sup over the $\mu_{i,l}'$ as well.
\end{remark}

\begin{lemma}\label{lem:proj} (Strong Small Eigenvalue Lemma)
Let $X\in \C^n$ be a random vector with distribution $f_X:\C^n\to [0,\infty)$ and for $W\in G_{n,n-d}$ chosen at random. Then,
\[
\P(\norm{\proj_W(X)}\leq \delta)\leq \omega_d \delta^d \norm{f_X}_{d,\delta}.
\]
(Here the probability on the left is over both $X$ and $W$.)
In particular, if $X_1,\dots,X_n$ are rows of an $n\times n$ matrix, with possibly distinct distributions $f_{X_i}$, and $\sigma$ is a permutation of $\set{1,\dots,n}$
Then for $1\leq j\leq d$,
\[
\P(\norm{\proj_W(X_{\sigma(j)})}\leq \delta)\leq \omega_d \delta^d 
\norm{f_{X_1},\dots,f_{X_{n}}}_{d,\delta}.
\]
where the left hand side is the probability over $X_{\sigma(j)}$ and all $(n-d)$-subspaces $W$ 
of the form $W=X_{\sigma(d+1)}\wedge \cdots\wedge X_{\sigma(n)}$ with their corresponding 
induced distribution.
\end{lemma}

\begin{proof}
Taking expectations in $W$ on both sides of the inquality from Lemma \ref{lem:proj1} we have
\begin{align*}
\omega_d \delta^d \E_W[\norm{f_X}_{W,\delta,1}]&\geq \E_W[\P_{X|W}(\norm{\proj_W(X)}\leq \delta|W)]\\
&=\P_{X,W}(\norm{\proj_W(X)}\leq \delta),
\end{align*} 
where the last equality follows by Fubini. 

The last statement follows from taking $X=X_{\sigma(j)}$ and the distribution on $\Gr_{n,n-d}$ induced by the map 
$(X_{\sigma(d+1)}, \dots, X_{\sigma(n)})\mapsto W=X_{\sigma(d+1)}\wedge \cdots\wedge X_{\sigma(n)}$ (which is well defined off of a measure zero subset).
\end{proof}

If $\omega_d \delta^d \norm{f_X}_{d,\delta}\leq C(n,\delta)$ we will need to bound the case that $\delta=n^{-\frac52-\beta}$ for some small $\beta$, by the decay rate $C(n,n^{-\frac52-\beta})\leq Cn^{-\frac52}$ for some universal constant $C$.  (See Theorem \ref{thm:leastsing} below.)

The following theorem demonstrates that when $\norm{f}_{d, \delta}$ is suitably bounded, then the singular values of a random matrix with iid rows drawn from $f$ are almost surely nonzero. \\

\begin{theorem}\label{thm:leastsing} (Least Singular Value):\\
Let $A_n$ be a random matrix with rows $X_i$ drawn from multivariate distributions $f_i$ and suppose that $\norm{f_1,\dots,f_n}_{1, \delta}$ is bounded for $\delta=\frac{1}{n^\frac{5}{2}\log(n)}$ independently of $n$. Then the smallest singular value of $A_n$ is almost surely greater than $\frac{1}{n^\frac{5}{2}\log(n)}$ as $n\to\infty$. \\
\end{theorem}

\begin{proof}
Denote the rows of $A_n$ by $X_i$ and lowest singular value of $A_n$ by $\sigma_n$.
Recall that 
$$ \sigma_n \geq \frac{1}{\sqrt{n}} \min_{i\leq n} ( \dist(X_i, H_i) ) $$
where $H_i = span\{X_j\}_{j \neq i}$. Thus, given $\delta > 0$, 

\begin{align*}
\P(\sigma_n \leq \delta) &\leq \sum_{i=1}^n \P( \dist(X_i, H_i) \leq \delta \sqrt{n}) \\
&= \sum_{i=1}^n \E_{H_i}[\P_{X_i|H_i}(\dist(X_i, H_i) \leq \delta \sqrt{n}|H_i)] \\
&\leq n \max_i \E_{H_i}[\P_{X_i|H_i}(\dist(X_i, H_i) \leq \delta \sqrt{n}|H_i)] \\
&\leq n\max_i E_{W\in \Gr_{n,n-1}}\left[\P_{X_i|W}(\lVert \proj_{W^\perp}(X_i) \rVert \leq \delta \sqrt{n}|W)\right]\\
&\leq n^{\frac{3}{2}} \delta \max_i E_{W\in \Gr_{n,n-1}} [\norm{f_i}_{W,\delta,1}]\\
&\leq n^{\frac{3}{2}} \delta \omega_1 \lVert f_1,\dots,f_n \rVert_{1, \delta}\\
\end{align*}
Note here that the measure on the Grassmanian on the fourth and fifth line will depend on $i$ unless the vectors are identically distributed.
 
Setting $\delta = \frac{1}{n^\frac{5}{2}\log(n)}$, the last line is $O(\frac{1}{n\log(n)})$ and the result follows from the Borel-Cantelli Lemma.
\end{proof}

\begin{remark}
Note that the above theorem does not rely on the independence of the rows, but the case where each $f_i$ limits sufficiently fast to the same single dirac measure shows that the condition on the norm is an essential hypothesis.
\end{remark}

We now give a family of examples where the rows are equidistributed, i.e. $f_i=f_X$, and $\norm{f_X}_{1,\delta,1}\geq \frac{C}{\delta^d
\omega_d}$, but where we still have $\norm{f_X}_{d,\delta}<C$ uniformly in $n$ provided  $\delta=o(n^{-1})$. Hence the asymptotic bound on the lowest singular value still applies.

\begin{example}
Suppose that $X$ has iid entries each of which is a Bernoulli variable distribution with point
masses of weight $p$ at $0$ and $(1-p)$ at $1$. The resulting distribution in $\C^n$ is supported on
the vertices of the $n$-cube with side-length $1$ whose vertices are all binary vectors of length
$n$. Now consider any choice of codimension $d$ plane $W^\perp$ which passes through the origin and
is spanned by any choice of $n-d$ distinct coordinate vectors $e_i$. (Note these are admissible in
that each has a positive probability of being a row of a random $n\times n$ matrix.) In this case,
$W^\perp$ contains a total of $2^{n-d}$ vertices, for a total mass of $p^d$. In this case,
$\norm{f_X}_{1,\delta,1}\geq \norm{f_X}_{W,\delta,1}=\frac{p^d}{\omega_d\delta^d}$, which explodes as $\delta\to 0$. Hence we cannot achieve a useful bound for this ``worst case'' choice of $W^\perp$. 

On the other hand, for the case $p=\frac{1}{2}$ at least, the main result of \cite{Tao07} states that the probability that a random matrix with iid $p=\frac{1}{2}$-Bernoulli $\{ -1,1 \}$-entries is singular is $(\frac{3}{4} +o(1))^n$. A standard procedure using row and column operations produces from an $(n+1)\times(n+1)$ $\{-1,1 \}$-matrix a new matrix with $n\times n$ $\{0,1 \}$-lower submatrix and first column $e_1$ (see e.g. \cite{Orrick05}). Hence the probability that a random iid $\frac{1}{2}$-Bernoulli $\{0,1 \}$-matrix is singular is also $(\frac{3}{4}+o(1))^{n+1}\leq(\frac34+o(1))^n$. 

A plane $P$ spanned by linearly independent $\set{0,1}$-vectors $v_1,\dots,v_{n-d}$ contains the $\set{0,1}$-vector $v$ if and only if for all choices of $\set{0,1}$-vectors $w_1,\dots,w_{d-1}$ the $n\times n$ $\set{0,1}$-matrix $A_n=[v_1,\dots,v_{n-1},v,w_1,\dots,w_{d-1}]$ is singular.

The condition that $v_1\wedge\dots\wedge v_{n-d}\neq 0$ is the condition that there be $\set{0,1}$-vectors $w_1,\dots,w_d$ such that $B_n=[v_1,\dots,v_{n-d},w_1,\dots,w_d]$ be nonsingular. Note that a subset of this last condition is the case that the $(n-d)\times(n-d)$ minor, denoted $C_{n-d}$, be nonsingular. The probability that $C_{n-d}$ be nonsingular is the same that $A_{n-d}$ be nonsingular. So the probability that $v\in P$ for a randomly chosen plane $P\in \Gr_{n-d}(\C^n)$ can be estimated as
\begin{align*}
\P( v\in P) &\leq \frac{\P(\det(A_n)=0)}{1-\P(\det (B_n)= 0\ \forall w_1,\dots,w_d)}\leq  \frac{\P(\det(A_n)=0)}{\P(\det (A_{n-d})\neq 0)}\\
&\leq \frac{(\frac34+o(1))^n}{1-(\frac34+o(1))^{n-d}}=(\frac34+o(1))^n.
\end{align*}
Now if we let $P=W^\perp$ be the plane spanned by $n-d$ randomly chosen rows of our random $n\times n$ matrix we  note that the nearest distance to the plane $W^\perp$ of a vertex not in the plane is the distance of the origin to the standard $n$-simplex, namely $\frac1n$. Hence, letting $W^\perp$ vary over all choices of rows, and for $\delta<\frac1n$ we obtain that $\norm{f_X}_{d,\delta}\leq (\frac34+o(1))^n$. Since we will be taking $\delta<n^{-\frac52}$ we obtain the desired bound. Note the same estimate still holds even when $d$ is allowed to grow in $n$ provided $d=o(n)$.

Lastly we consider a case where we allow coordinate-wise dependency. Suppose the random Bernoulli vector $X$ has coordinate wise probability $\frac12$ of being $0$ or $1$ but has symmetric $n\times n$ covariance matrix $\op{Cov}(X)=[C_{ij}]$ with entries $C_{i,j}\in [-\frac14,\frac14]$ and $C_{ii}=\frac14$, i.e. the joint pairwise probabilities on coordinates $i$ and $j$ being $1$ are $p_{ij}=C_{i,j}+\frac14$, but are otherwise independent. The vector $W\cdot X$, where $W$ is the Whitening Matrix such that $W^tW=C^{-1}$, has covariance the identity matrix. Provided the entries of the off-diagonal entries $C_{ij}$ are uniformly bounded away from $\frac14$, and $\delta=o(\frac{1}{n})$, then we will have identical probability of $X$ being in the $\delta$-neighborhood of $W^{-1}P$ as for $WX$ in some $O(\delta)$-neighborhood of $P$. In particular $\norm{f_X}_{d,\delta}\leq (\frac34+o(1))^n$ provided $\delta=o(\frac1{n})$.
\end{example}

\section{Generalized Mar\v{c}enko-Pastur}

In this section, we discuss a key ingredient of the Circular Law: a corresponding limiting law for symmetric random matrices. Given $A_n$, the shifted and symmetrized version of $A_n$ is 
$$ H_n = H_n(z) = (\frac{1}{\sqrt{n}}A_n - zI)(\frac{1}{\sqrt{n}}A_n -zI)^* $$
for $z \in \mathbb{C}$. \\ 

One desires that the spectral distribution of $H_n$ converges in the large $n$ limit to a deterministic measure, independent of the entries of $A_n$. Or, as in our case, independent of the probability measure in $\C^n$ (resp. $\C^n$) that the rows of $A_n$ are independently drawn from. \\

The classical version of this result is the Mar\v{c}enko-Pastur law and does not assume that $A_n$ is square to begin with. Moreoever, the Mar\v{c}enko-Pastur law deals with iid entries and with the the partical case of shifting by $z=0$. We state it here for the reader.

\begin{theorem} (Mar\v{c}enko-Pastur Law)
Let $A_n$ be a $n \times N_n$ random matrix with iid entries that satisfy $\E[x_{ij}] = 0$ and $\E[x_{ij}^2] = 1$ and suppose that $\frac{p}{n} \rightarrow c \in (0, 1]$ as $n \rightarrow \infty$. Denote by $\mu_n$ the spectral distribution of $\frac{1}{n}A_nA_n^{*}$. Then, $\mu_n \rightarrow \mu$ almost surely, where $\mu$ is a deterministic measure given by 
$$ \frac{d\mu}{dx} = \frac{1}{2\pi xy} \sqrt{(b-x)(x-a)}\mbb{1}_{a \leq x \leq b} $$
where $a = (1 - \sqrt{y})^2$ and $b = (1+\sqrt{y})^2$
\end{theorem} 

\vspace{1pc}

The Mar\v{c}enko-Pastur Law has many methods of proof: combinatorial, methods using the Stieltjes (Cauchy) transform, and methods using free probability. The combinatorial proof is the most readily generalized to our setting and has been utilized by other authors to extend the Mar\v{c}enko-Pastur Law to have certain amounts of dependency (Adamczak) and we will follow similar suit, but with some different assumptions. \\

Let $(N_n)_{n \geq 1}$ be a sequence of postive integers such that $\lim_{n \rightarrow \infty} n/N_n = y \in (0, \infty)$. Recall we have the following assumptions: \\

\begin{enumerate}
\item[(A1)] for every $k\in \mathbb{N}$, $\sup_n \max_{1\leq i\leq n,1\leq j\leq N_n} \E[|x_{ij}^{(n)}|^k] < \infty$;

\item[(A2)] for every $k\in \mathbb{N}$, the sum of all terms in $\Mom_{2k}(A_n)$ with at least one $x_{ij}$ appearing with a power of 1 is of size $o_k(n^{k+1})$;

\item[(A3)] for every $\epsilon > 0$:
$$ \lim_{n\rightarrow \infty} \frac{1}{n} \sum_{i \leq n} 
\P\( \abs{ \frac{1}{N_n}\sum_{j=1}^{N_n} (x_{ij}^{(n)})^2 - 1 } \geq \epsilon \) = 0 $$

and 

$$ \lim_{n\rightarrow \infty} \frac{1}{N_n} \sum_{j \leq N_n} 
\P\(\abs{\frac{1}{n}\sum_{i=1}^{n} (x_{ij}^{(n)})^2 - 1 } \geq \epsilon \) = 0.$$
\end{enumerate}

We need the following theorem for proof of the Circulaw Law. We will defer the proof of this result to Subsection \ref{sec:Graphs}.\\

\begin{theorem}[cf. Theorem 2.4 \cite{Adamczak11}]\label{thm:H_n} Assume that $N_n = n$ and $A_n$ is a sequence of random matrices with rows independently drawn from a probability measure in $\C^n$ given by $f^{(n)}$. Assume that $A_n$ satisfies assumption A1-A3. Then for any $k \in \mathbb{N}$, 
$$ \lim_{n\rightarrow \infty} \frac{1}{n}\E[\tr H_n^k] = \mu_k(|z|^2), $$
where $\mu_k(|z|^2)$ is a function depending only $|z|^2$ and not on the distribution of $H_n$. 
\end{theorem}

\begin{corollary}[cf. Corollary 2.5 \cite{Adamczak11}]\label{cor:Lweakly}
Let $A_n$ be as in Theorem \ref{thm:H_n} and let $L_n(z)$ be the spectral measure of $H_n(z)$. For every $z \in \C$, $L_n(z)$ converges weakly to a non-random measure which does not depend on the distribution of the rows of $A_n$. 
\end{corollary}
\begin{proof}
The proof of this corollary is verbatim identical to the proof of Cor 2.5 of  \cite{Adamczak11} except for the replacement of Theorem 2.4 of \cite{Adamczak11} with Theorem \ref{thm:H_n} above.
\end{proof}

\subsection{Combinatorial Trees}

\vspace{1pc}
In what follows, we will use certain combinatorial structures to keep track of distinct classes of terms. A detailed background for these structures can be found in Chapter 3 of \cite{Bai10}. We will closely follow the notation of \cite{Adamczak11}.
 
Let $T = (V, E, r)$ be a rooted tree. A $\Gamma$-tree is a rooted tree having the following structure:
\begin{itemize}
\item The set $V$ is partitioned into two sets $S$ and $O$, denoting special and ordinary vertices

\item Every edge adjacent to a special vertex is given an orientation so that
\begin{itemize}
\item For any $u, w \in S$ such that on the path $u = v_0v_1..v_m = w$ connecting $u$ and $w$, we have that $v_1, .., v_{m-1} \in O$. If $m$ is odd, then the orientations of the first and the last edge on this path are the same. That is to say, one has ($u \rightarrow v_1$ and $v_{m-1} \rightarrow w$) or ($v_1 \rightarrow u$ and $w \rightarrow v_{m-1}$). If $m$ is even, then the orientation of the first and the last edge in the path are opposite. 
\item if $r \in O$, then for any $u\in S$ such that $u$ is the only special vertex on the path $r = v_0v_1...v_m = u$, one has $v_{m-1} \rightarrow u \iff m$ is odd. 
\end{itemize}
\end{itemize}

\vspace{1pc}

Given the orientation of paths between special vertices, we can partition $O$ into two sets $U$ and $D$. Let $u \in O$ and $r = v_0v_1...v_m = u$ be a path from the root $r$ to $u$. 
\begin{itemize}
\item if $r \in O$ and $v_1, ..., v_{m-1} \in O$, then $u \in D \iff m$ is odd. Otherwise, $u \in U$. 
\item if $v_l$ is the last special vertex on the path, then $u\in D \iff$ ($m-l$ is odd and $v_l \rightarrow v_{l+1}$) or ($m-l$ is even and $v_{l+1} \rightarrow v_l$).
\end{itemize}

\vspace{1pc}

Note that every edge which has ends that are ordinary vertices must have one end in $D$ and one end in $U$ and we can assign to each edge an orientation $u \rightarrow v$ where $u\in U$ and $v \in D$. We write $e = (u\rightarrow v)$.\\

Let $A_n$ be a sequence of random $n\times n$ matrices with $A_n = [x_{ij}^{(n)}]$ and let $T$ be a $\Gamma$-tree. Let $I_T^n$ be the set of functions $\mbf{i} = (i_v)_{v \in V} : V \rightarrow \set{1, ..., n}$ such that if $A$ is one of the sets $D \cup S$, $U \cup S$ then $\forall u,v \in A$, $u \neq v \implies i_u \neq i_v$. And, for every $e = (u \rightarrow v) \in E$, $i_u \neq i_v$. \\

We define 
$$\xi_n(T) = n^{-|E| -1} \E \left( \sum_{\mbf{i} \in I_T^n} \prod_{e = (u \rightarrow v)\in E} ( x_{i_ui_v}^{(n)})^2 \right) $$

We will prove the following proposition necessary to prove Theorem \ref{thm:H_n}. 

\begin{proposition}\label{prop:T=1}
Let $N_n = n$ and let $A_n$ be as in Theorem \ref{thm:H_n}. Then for every $\Gamma$-tree $T$,
$$ \lim_{n\rightarrow \infty} \xi_n(T) = 1 $$
\end{proposition}

\begin{proof}
We prove by induction on the size of the tree. If $|V| = 1$ then it is clear that $\xi_n(T) = 1$ for all $n$. \\

Suppose that the proposition holds for all trees of size $s \leq m-1$ and that $T$ is a tree of size $m$. Consider an arbitrary leaf $w$ of the tree $T$, where $w$ is not the root and let $x$ be the unique neighbor of $w$. \\

We consider the case where $w \in D$ (The proof when $w \in U$ follows similarly so we omit it). Let $\tilde T = (\tilde V, \tilde E, r)$ be the tree obtained from $T$ by deleting $w$ along with the edge $e = (x \rightarrow w)$.  \\

Let $\tilde I^n$ be the set of multi-indices $\mbf{i}_{\tilde V}: \tilde V \rightarrow \{ 1, ..., n \}$ which can be extended to a mutli-index $\mbf{i}_V = (\mbf{i}_{\tilde V}, i_w) \in I_T^n$. Denote $U(\mbf{i}_{\tilde V}) = \prod_{e = (u \rightarrow v) \in \tilde E} (x_{i_u i_v}^{(n)})^2$. We have
$$ \xi_n(T) = n^{-|U| - |D|}  \sum_{\mbf{i}_{\tilde V} \in \tilde I^n} \sum_{i_w : (\mbf{i}_{\tilde V}, i_w) \in I_T^n} \E \left( ( x_{i_xi_w}^{(n)})^2 U(\mbf{i}_{\tilde V}) \right) $$

For large enough $n$, $I_{\tilde T}^n = \tilde I^n$, and there are only $|D|-1$ choices for $i_w$ such that $(\mbf{i}_{\tilde V}, i_w) \notin I_T^n$. 
By A1, we have that $\E((x_{ij}^{(n)})^2)$ is bounded for all $i, j$, independent of $n$. Thus, by generalized H\"{o}lder's inequality, for every such $i_w$ we have that $\E( (x_{i_xi_w})^2 U(\mbf{i}_{\tilde V}))$ is bounded by a number independent of $n$. Thus, for large enough $n$, we have

\begin{align*}
 \xi_n(T) &= n^{-|U| - |D|}  \sum_{\mbf{i}_{\tilde V} \in I_{\tilde T}^n} \left( \sum_{i_w = 1}^n \E \left( ( x_{i_xi_w}^{(n)})^2 U(\mbf{i}_{\tilde V})\right) + O_T(1) \right) \\
 &= n^{-|U| - |D|}  \sum_{\mbf{i}_{\tilde V} \in I_{\tilde T}^n}  \sum_{i_w = 1}^n \E \left( ( x_{i_xi_w}^{(n)})^2 U(\mbf{i}_{\tilde V}) \right) + o_T(1)
\end{align*}
where the constant depends on $T$ and where in the last inequality we use the fact that $|I_{\tilde T}^n| = n(n-1)...(n-|U|+1)n...(n-|D|+2) = O_T(n^{|U| + |D|-1})$.\\

Notice that, 
\begin{align*}
\left|n^{-|U| - |D|} \sum_{\mbf{i}_{\tilde V} \in I_{\tilde T}^n}  \sum_{i_w = 1}^n \E \left( ( x_{i_xi_w}^{(n)})^2 U(\mbf{i}_{\tilde V}) \right)  - \xi_n(\tilde T)  \right| \\
\leq n^{-|U|-|D|+1} \sum_{\mbf{i}_{\tilde V} \in I_{\tilde T}^n} \E \left( \abs{n^{-1} \sum_{i_w = 1}^n ( x_{i_xi_w}^{(n)})^2 - 1 } U(\mbf{i}_{\tilde V}) \right). 
\end{align*}
Moreover, using that $y\leq \max\set{\eps,y}$ for any $\eps>0$ and then applying Cauchy-Schwarz, we have for every $\mbf{i}_{\tilde V} \in I_{\tilde V}^n$ that,
\[
\E \left( \abs{n^{-1} \sum_{i_w = 1}^n ( x_{i_xi_w}^{(n)})^2 - 1 } U(\mbf{i}_{\tilde V}) \right) \leq \\
\epsilon \E U(\mbf{i}_{\tilde V}) + \norm{ \abs{n^{-1} \sum_{i_w = 1}^n ( x_{i_xi_w}^{(n)})^2 - 1 } U(\mbf{i}_{\tilde V})}_2
\P\left(\abs{n^{-1}\sum_{j=1}^n (x_{i_xj}^{(n)})^2 - 1} \geq \epsilon \right)^{1/2}
\]

\vspace{1pc}

By A1, the triangle inequality in $L_p$, and generalized Holder's inequality, we have that
$$ \E U(\mbf{i}_{\tilde V}) \text{ and } \norm{ \abs{n^{-1} \sum_{i_w = 1}^n ( x_{i_xi_w}^{(n)})^2 - 1 } U(\mbf{i}_{\tilde V})}_2 $$
are bounded by some constant $C$, depending only on $T$ and the bounds from A1. Thus, we get that

$$ |\xi_n(T) - \xi_n(\tilde T)| \leq \epsilon + C n^{-|U|-|D|+1} \sum_{\mbf{i}_{\tilde V} \in I_{\tilde T}^n} (\epsilon + \P(|\sum_{j=1}^n (x_{i_xj}^{(n)})^2 - n| \geq \epsilon n)^{1/2}) $$

Since for each $i_x$ there are at most $n^{|U|+|D|-2}$ multi-indices $\mbf{i}_{\tilde V \setminus \{x\} }$ such that $\mbf{i}_{\tilde V} = (\mbf{i}_{\tilde V \setminus \{ x \}}, i_x) \in I_{\tilde T}^n$, we get that

$$ |\xi_n(T) - \xi_n(\tilde T)| \leq (C+1)\epsilon + \frac{C}{n} \sum_{i_x=1}^n \P(|\sum_{j=1}^n (x_{i_xj}^{(n)})^2 - n| \geq \epsilon n)^{1/2} $$

Using the Cauchy-Schwartz inequality and the first part of assumption A3 (note: for $w \in U$, we simply use the second part of assumption A3 here), we get that  
$$ \xi_n(T) - \xi_n(\tilde T) = o_T(1) $$

Thus, we have that
$$\lim_{n \rightarrow \infty} \xi_n(T) = 1$$ for trees of size $m$ and the proof follows by induction. 

\end{proof}

\subsection{$\Delta$ Graphs}\label{sec:Graphs}

To prove Theorem \ref{thm:H_n} we begin by introducing the notion of $\Delta$ graphs. Here we follow closely the work outlined in Adamczak (\cite{Adamczak11}) with slight modifications for our alternative assumptions.\\

For two sequences of integers $\mbf{i} = (i_1, ..., i_k)$ and $\mbf{j} = (j_1, ..., j_k)$, we define a $\Delta$-graph $\Delta=G(\mbf{i}, \mbf{j})$ as a bipartite graph $(I_i, I_j, E)$ such that $I_i = \{ i_1, ..., i_k \}$ (the upper indicies) and $I_j = \{ j_1, ..., j_k \}$ (the lower indices) and the set $E$ of edges consisting of $k$ directed edges from $i_u$ to $j_u$ and $k$ directed edges from $j_u$ to $i_{u+1}$, where we set $i_{k+1} = i_1$. We also label the edges from 1 to $2k$ in the order of $(i_1, j_1), (j_1, i_2),(i_2, j_2),..., (i_k, j_k),(j_k, i_1)$. Note that $I_i$ and $I_j$ may not be disjoint, but their common elements are treated as different objects when considered as upper and lower vertices of the graph. \\

We would also like to partition into classes of up and down edges. An edge will be called perpindicular if its two end vertices are equal and skew if they are distinct. For any $\Delta$-graph $\Delta$, let $UP(\Delta)$ denote the up edges, $DP(\Delta)$ denote the down edges, and $S(\Delta)$ denote the skew edges.\\

\begin{definition}
Pairs $(\mbf{i} , \mbf{j})$ and $(\mbf{i}^{\prime}, \mbf{j}^{\prime})$ are isomorphic if there exist functions $f: I_{i} \rightarrow I_{i^{\prime}}$ and $g: I_{j} \rightarrow I_{j^{\prime}}$, such that for $u = 1, ..., k$ one has:
\begin{itemize}
\item $f(i_u) = i_u^{\prime}, g(j_u) = j_u^{\prime}$
\item $f(i_u) = g(j_u) \iff i_u = j_u$
\item $f(i_{u+1}) = g(j_u) \iff i_{u+1} = j_u$
\end{itemize}
\end{definition}

\begin{definition}
$G(\mbf{i}, \mbf{j})$ and $G(\mbf{i}^{\prime}, \mbf{j}^{\prime})$ are isomorphic if and only if $(\mbf{i},\mbf{j})$ and $(\mbf{i}^{\prime}, \mbf{j}^{\prime})$ are ismorphic. We write $G(\mbf{i}, \mbf{j})\sim G(\mbf{i}^{\prime}, \mbf{j}^{\prime})$ when the two graphs are isomorphic.
\end{definition}

Let $\Delta(k)$ be a set of representatives of isomorphism classes of $\Delta$-graphs $G(\mbf{i},\mbf{j})$ with $\mbf{i} = (i_1, ..., i_k)$, $j = (j_1, ..., j_k)$, and $i_l, j_l \in \{ 1, ..., 2k \}$. Any graph based on two sequences of length $k$ is isomporphic to a graph in $\Delta(k)$. \\

\begin{definition}
Given $\Delta \in \Delta(k)$, we definite $I_{\Delta}^n$ to be the set of all indices $\mbf{i}: V(\Delta) \rightarrow \{ 1, ..., n \}$ such that
\begin{itemize}
\item for any two upper indices $v,w$, we have $i_v \neq i_w$
\item for any two lower indices $v,w$, we have $i_v \neq i_w$
\item for any edge $i_{u(e)} = i_{d(e)} \iff e$ is perpindicular
\end{itemize}
\end{definition}

Denote by $W_n = \frac{1}{\sqrt{n}}A_n - zI = (w_{ij})$, so that $H_n=W_nW_n^*$ where we now supress the dependence on $n$ when we write the entries for ease of notation; we will also suppress this dependence for the entries $x_{ij}$. We now prove Theorem \ref{thm:H_n}.

\begin{proof} (of Theorem \ref{thm:H_n})
We have
\begin{align*}
&\frac{1}{n} \E[ \tr H_n^k] \\
&= \frac{1}{n} \sum_{i_1, ..., i_k = 1}^{n} \sum_{j_1, ..., j_k = 1}^{n} \E w_{i_1j_1} \bar w_{i_2j_1}...w_{i_kj_k} \bar w_{i_1j_k} \\
&= \frac{1}{n} \sum_{\Delta \in \Delta(k)} \sum_{\substack{i, j \in \{ 1, ...,n \}^k :  \\ G(i, j) \sim \Delta}} \E w_{i_1j_1} \bar w_{i_2j_1}...w_{i_kj_k} \bar w_{i_1j_k} \\
&= \sum_{\Delta \in \Delta(k)} \frac{1}{n} \sum_{\mbf{i} \in I_{\Delta}^n} \E \left( \prod_{e \in S(\Delta)} w_{i_{u(e)}i_{d(e)}}  \prod_{e \in UP(\Delta)} \bar w_{i_{u(e)}i_{d(e)}} \prod_{e \in DP(\Delta)} w_{i_{u(e)}i_{d(e)}} \right) \\
&= \sum_{\Delta \in \Delta(k)} \frac{1}{n^{\alpha}} \sum_{\mbf{i} \in I_{\Delta}^n} \E \left( \prod_{e \in S(\Delta)} x_{i_{u(e)}i_{d(e)}}  \prod_{e \in UP(\Delta)} \bar w_{i_{u(e)}i_{d(e)}} \prod_{e \in DP(\Delta)} w_{i_{u(e)}i_{d(e)}} \right) \\
\end{align*}
where $\alpha = {1+|S(\Delta)|/2}$. For a fixed $\Delta$, let $\Delta^{\prime}$ be the graph obtained by replacing each pair of vertices connected with a perpendicular edge by one vertex and removing all corresponding perpendicular edges, while keeping all skew edges so that $\Delta^{\prime}$ is connected and has $|S(\Delta)|$ edges. For this fixed $\Delta$, each term in the sum over $\mbf{i} \in I_{\Delta}^n$ above is bounded by some constant in $k, z$ as we have all bounded moments of the individual $x_{ij}$.\\

As $|I_{\Delta}^n| \leq n^{|V(\Delta^{\prime})|}$, the graphs $\Delta$ such that $\Delta^{\prime}$ has fewer than $1 + S(\Delta)/2$ vertices have no asymptotic contribution. Note that in the case of $z=0$, these are entries of $\Mom_k(A_{n})$ that have $x_{ij}$ terms with powers greater than or equal to 2, but not all equal to 2. \\

Moreover, for skew edges $e = (v,w)$ of multiplicity 1, $(w,v)$ is not an edge of $\Delta$ and so the corresponding variable $x_{i_{u(e)}j_{v(e)}}$ appears in the product exactly once. By assumption A2, the sum of these terms is also asymptotically negligible. Note that in the case of a random matrix with mean zero iid entries, these terms vanish automatically.\\

We are left with the graphs $\Delta$ for which each skew edge $e$, treated as an undirected edge, appears only twice and $\Delta^{\prime}$ has at least $1 + |S(\Delta)|/2$ vertices. Let $\Delta^{\prime\prime}$ be the graph formed by  identifying up and down edges of $\Delta^{\prime}$ that share the same endpoints. This implies that number of edges $a$ of $\Delta^{\prime\prime}$ is at most $|S(\Delta)|/2$. If $b$ is the number of vertices of $\Delta^{\prime\prime}$, we have that $b \geq a+1$. Moreover, since $\Delta^{\prime\prime}$ is connected we have that $b = a+1 = |S(\Delta)|/2 + 1$, and $\Delta^{\prime\prime}$ is a tree. Since the cycle in $\Delta^{\prime}$ inherited from $\Delta$ corresponds to a walk in $\Delta^{\prime\prime}$ which goes through every vertex and returns to the starting vertex, it means all skew edges in $\Delta$ appear exactly twice. We also have that among the perpendicular edges connected any two vertices of $\Delta$, there are equal numbers of up and down edges. Thus, we can write
$$
\frac{1}{n} \E trH_n^k = \sum_{\substack{ \Delta \in \Delta(k): \\ \Delta^{\prime\prime} \text{is a tree}}} |z|^{2|UP(\Delta)|} n^{-\alpha} \sum_{i \in I_{\Delta}^n} \E \prod_{e \in S(\Delta)} x_{i_{u(e)}i_{d(e)} } + o_{k,z}(1) $$

To each $\Delta$ such that $\Delta^{\prime\prime}$ is a tree, we can assign a $\Gamma$-tree $T(\Delta)$, where the special vertices are obtained by merging vertices of $\Delta$ connected by perpindicular edges and the orinetation of edges is always from up to down. Using Proposition \ref{prop:T=1}, we have

\begin{align*}
& \frac{1}{n} \E \tr H_n^k \\
&= \sum_{\substack{ \Delta \in \Delta(k): \\ \Delta^{\prime\prime} \text{is a tree}}} |z|^{2|UP(\Delta)|} \xi_n(T(\Delta)) + o_{k,z}(1) \\
&= \sum_{\substack{ \Delta \in \Delta(k): \\ \Delta^{\prime\prime} \text{is a tree}}} |z|^{2|UP(\Delta)|} + o_{k,z}(1)
\end{align*}

which completes the proof.
\end{proof}

\section{Proof of Circular Law}

In this section we prove the main theorem of this paper

\begin{theorem}\label{thm:weakly}
Let $A_n$ be a sequence of $n \times n$ random matrices with independent rows $X_1^{(n)}, ..., X_n^{(n)}$ defined on a common probability space and satisfying assumptions A1-A3. Assume that for each $n$ and $i,d\leq n$, and all $\delta>0$, the probability measures $\nu_i^{(n)}$ for $X_i^{(n)}$ have uniformly bounded $\norm{\nu_1^{(n)},\dots,\nu_n^{(n)}}_{d, \delta}$. Then almost surely the spectral measure of $\mu_{\frac{1}{\sqrt{n}}A_n}$ converges weakly to the uniform distribution on the unit disk in $\C$. 
\end{theorem}

To prove this theorem, we will use the following replacement principle for random matrices by Tao and Vu.

\begin{theorem}[Tao-Vu Replacement Principle]\label{thm:tau-vu}
Suppose for each $n$ that $A_n, B_n \in M_n(\C)$ are ensembles of random matrices defined on a common probability space. Assume that
\begin{enumerate}
\item
$$\frac{1}{n^2}\norm{A_n}^2 + \frac{1}{n^2}\norm{B_n}^2$$ is almost surely bounded
\item
for almost all complex numbers $z$,
$$ \frac{1}{n}\log|\det(\frac{1}{\sqrt{n}}A_n - zI)| - \frac{1}{n}\log|\det(\frac{1}{\sqrt{n}}B_n - zI)|$$
coverges almost surely to zero.
\end{enumerate}
Then $\mu_{\frac{1}{\sqrt{n}}A_n} - \mu_{\frac{1}{\sqrt{n}}B_n}$ coverges almost surely to 0.
\end{theorem}

To use the replacement principle, note that if the rows of random matrices are distributed according to the multivariate Gaussian random variables with independent coordinates having mean zero and finite second moment, then the $\norm{\nu_1^{(n)},\dots,\nu_n^{(n)}}_{d, \delta}$ is bounded and assumptions A1-A3 are satisifed. As matrices of this kind are known to have spectral distribution converging to the uniform distribution on the unit disk in $\C$, we need only show that the replacement principle holds for matrices of the kind described in Theorem \ref{thm:weakly}. \\

We first verify the second condition in Theorem \ref{thm:tau-vu}. We wish to show that for any $z \in \C$, with probability one, 
$$ \frac{1}{n}\log|\det(\frac{1}{\sqrt{n}}A_n - zI)| - \frac{1}{n}\log|\det(\frac{1}{\sqrt{n}}B_n - zI)| \rightarrow 0$$
Denote the rows of $\frac{1}{\sqrt{n}}A_n - zI$ by $Z_1, Z_2, ..., Z_n$ and the rows of $\frac{1}{\sqrt{n}}B_n - zI$ by $Y_1, Y_2, ..., Y_n$. Denote by $V_i$ the span of $Z_1, Z_2, ..., Z_{i-1}$ and by $U_i$ the span of $Y_1, Y_2, ..., Y_{i-1}$. We then have that
$$ \frac{1}{n}\log|\det(\frac{1}{\sqrt{n}}A_n - zI)| = \frac{1}{n} \sum_{i=1}^n \log\dist(Z_i, V_i)$$
and 
$$ \frac{1}{n}\log|\det(\frac{1}{\sqrt{n}}B_n - zI)| = \frac{1}{n} \sum_{i=1}^n \log\dist(Y_i, U_i)$$
and we wish to show that
$$ \frac{1}{n} \sum_{i=1}^n \log\dist(Z_i, V_i) - \frac{1}{n} \sum_{i=1}^n \log\dist(Y_i, U_i) \rightarrow 0 $$

Now, recall the following identity.

\begin{lemma}\label{lem:dist}
Let $1\leq k \leq n$ and $M$ be a full rank $k \times n$ matrix with singular values $\sigma_1(M) \geq ... \geq \sigma_k(M) > 0$ and rows $X_1, ..., X_k \in \C^n$. For each $1 \leq i \leq k$, let $W_i$ be the hyperplane generated by the $k-1$ vectors $X_1, ..., X_{i-1}, X_{i+1}, ..., X_k$. Then, 
$$ \sum_{j=1}^k \sigma_j(M)^{-2} = \sum_{j=1}^k \dist(X_j, W_j)^{-2}$$
\end{lemma}

By Theorem 2.15 and the Borel-Cantelli lemma, we have that 
$$ \log\dist(Z_i, V_i), \log\dist(Y_i, U_i) \geq -C_1 \log n$$
with probability one for some constant $C_1>0$. Furthermore, by A3 (or from Proposition \ref{prop:tail} proved independently below), we have that with probability one there exists a constant $C_2>0$ such that for large $n$ and $i \leq n$, 
$$ \log\dist(Z_i, V_i), \log\dist(Y_i, U_i) \leq C_2$$

Thus, to show the second condition of the replacement principle, it suffices to show that 
$$ \frac{1}{n} \sum_{i=1}^{n - n^{.99}} \log\dist(Z_i, V_i) - \frac{1}{n} \sum_{i=1}^{n-n^{.99}} \log\dist(Y_i, U_i) \rightarrow 0 $$

Following Tao and Vu in \cite{Tao10}, we show this with the following two lemmas.

\begin{lemma} (High-dimensional contribution)
There exists a constant $C$, such that for every $\epsilon \in (0, 1/4)$ and every $\delta \in (0, \epsilon^2)$, with probability one, for sufficiently large $n$, 
$$  \frac{1}{n} \sum_{(1-\delta)n \leq i \leq n- n^{.99}} (|\log\dist(Z_i, V_i)| + |\log\dist(Y_i, U_i)|) \leq C\eps$$
\end{lemma}

\begin{proof}
We only consider the first part of the sum involving the $Z_i's$, since the argument for the $Y_i$ terms is identical. We consider the positive and negative components of the logarithm separately. For the positive component, as $ \max(\log\dist(Z_i, V_i), 0)$ is bounded by the constant $C_2$, as mentioned earlier, we have with probability 1, for $\delta < \epsilon$,
$$  \frac{1}{n} \sum_{(1-\delta)n \leq i \leq n- n^{.99}} \max(\log\dist(Z_i, V_i), 0) \leq C_2\epsilon $$

To deal with the negative component of the logarithm, by the Borel-Cantelli lemma, it suffices to show
$$ \sum_{i=1}^{\infty} \P\left(\frac{1}{n} \sum_{(1-\delta)n \leq i \leq n- n^{.99}} \max(-\log\dist(Z_i, V_i), 0) \geq \epsilon\right) < \infty$$

To show this, we use the prove the following proposition.

\begin{proposition}[Lower tail bound]\label{prop:tail}
Let $1 \leq d \leq n$ and $c>0$, and let $V$ be a random d-dimensional subspace of $\C^n$. Let $X$ be a row of $A_n$. Then,
$$ \P(\dist(\sqrt{n}X, V) \leq c\sqrt{n-d}) = O(c^d\omega_d (1-\frac{d}{n})^{\frac{d}2})$$
where the constant depends on $c$. Moreover, the right hand side is less than $O((\frac{d}{2 e \pi c^2})^{-\frac{d}2})$.
\end{proposition}
\begin{proof}

Recall the definition of $\norm{f_{X_1},\dots,f_{X_n}}_{d,\delta}$ and the (non-measure preserving) homeomorphism given by $V\mapsto V^\perp$. Then we have,
\begin{align*}
\P_{X,V\in \Gr_{n,d}}(\dist(\sqrt{n}X, V) \leq c \sqrt{n-d} ) &\leq \P_{X,W\in \Gr_{n,n-d}}( \rVert \proj_W(\sqrt{n}X)\rVert \leq c \sqrt{n-d}) \\
&\leq \P_{X,W\in \Gr_{n,n-d}}\left( \rVert \proj_W(X)\rVert \leq c \sqrt{1-\frac{d}{n}}\right) \\
&\leq \omega_d \left(c \sqrt{1-\frac{d}{n}}\right)^{d} \norm{f_{X_1},\dots,f_{X_n}}_{d, c \sqrt{1-\frac{d}{n}}}\\
\end{align*}
The last inequality follows from Lemma \ref{lem:proj}. 

Since, $\norm{f_{X_1},\dots,f_{X_n}}_{d, \delta}<C$ for all $\delta>0$, we obtain the bound $O(\omega_d c^d(1-\frac{d}{n})^{\frac{d}2})$.

The last statement follows from the definition of $\omega_d$ followed by an application of Sterling's estimate.
\end{proof}

Now, since with probability one, $V_i$ is of dimension $i-1$, and $Z_i$ and $V_i$ are independent of each other, the proposition implies that 
$$ \dist(Z_i, V_i) \geq \sqrt{n-i +1}$$
for each $(1 - \delta)n \leq i \leq n-n^{.99}$, with probability $1- O(n^{-10})$, say. Setting $\delta$ sufficiently small, compared to $\epsilon$, taking logarithms and summing in $i$ and $n$, one obtains the result. 
\end{proof}

\begin{lemma} (Low-dimensional contribution)
There exists a constant $C$, such that for every $\epsilon \in (0, 1/4)$ and every $\delta_0 > 0$ such that for every $\delta \in (0, \delta_0)$, with probability one, for sufficiently large $n$, 
$$  \Bigl\lvert \frac{1}{n} \sum_{1 \leq i \leq(1-\delta)n} (\log\dist(Z_i, V_i) - \log\dist(Y_i, U_i)) \Bigl\rvert \leq C\epsilon$$
\end{lemma}

\begin{proof}
Let $n^{\prime} = \lfloor (1 - \delta) \rfloor$ and let $A_{nn^{\prime}}$ be the matrix with rows $\sqrt{n}Z_1, ..., \sqrt{n}Z_{n^{\prime}}$, and let $B_{nn^{\prime}}$ be the matrix with rows $\sqrt{n}Y_1, ..., \sqrt{n}Y_{n^{\prime}}$. Expressing the determinant as product of singular values, we have that
$$ \frac{1}{n} \sum_{1 \leq i \leq(1-\delta)n} \log\dist(Z_i, V_i) = \frac{1}{n} \sum_{i=1}^{n^{\prime}} \log(\sigma_i(A_{nn^{\prime}})) $$
and similarly for $Y_i$, $U_i$, and $B_{nn^{\prime}}$. Thus, it suffices to show that 
$$ \frac{1}{n^{\prime}} \sum_{i=1}^{n^{\prime}} \log(\sigma_i(A_{nn^{\prime}})) - \log(\sigma_i(B_{nn^{\prime}})) = O(\epsilon) $$
for all but finitely many $n$. 
Which amounts to showing that 
$$ \int_0^{\infty} \log t d\nu_{nn^{\prime}}(t) = O(\epsilon) $$
for all but finitely many $n$, where $d\nu_{nn^{\prime}} = \mu_{\frac{1}{n^{\prime}}A_{nn^{\prime}}A_{nn^{\prime}}^*} - \mu_{\frac{1}{n^{\prime}}A_{nn^{\prime}}A_{nn^{\prime}}^*}$. We show this by dividing the region of $t$ into parts.\\

\vspace{1pc}
(1) The region of very large $t$:\\
Note that 
\begin{align*}
\int_0^{\infty}  t |d\nu_{nn^{\prime}}(t)| &\leq \frac{1}{n^{\prime}} \sum_{i=1}^{n^{\prime}} (\frac{1}{n} \sigma_i(A_{nn^{\prime}})^2 + \frac{1}{n} \sigma_i(B_{nn^{\prime}})^2) \\
&= \frac{1}{n^{\prime}} \sum_{i=1}^{n^{\prime}} (|Z_i|^2 + |Y_i|^2) < C
\end{align*}
for some $C$, with probability 1, for sufficiently large $n$ and all $i$, by A3. Thus, we have that
$$ \int_{R_{\epsilon}}^{\infty} |\log t| |d\nu_{nn^{\prime}}(t)| \leq \epsilon$$
for all but finitely many $n$ and some $R_{\epsilon}$ depending only on $\epsilon$.
\vspace{1pc}

(2) The region of intermediate $t$:\\
Consider the region $t \in [\epsilon^4, R_{\epsilon}]$. First, recall the Cauchy Interlacing Property

\begin{lemma}[Cauchy Interlacing Property]\label{lem:Cauchy}
 Let $A$ be a $n\times n$ matrix with complex entries and let $A^{\prime}$ be the submatrix formed by the first $n-k$ rows. Let $\sigma_1(A), ..., \sigma_n(A)$ denote the singular values of $A$, and similarly for $A^{\prime}$. Then,
 $$ \sigma_i(A) \geq \sigma_i(A^{\prime}) \geq \sigma_{i+k}(A)$$
 for every $1 \leq i \leq n-k$
\end{lemma}

Let $\Psi$ be a smooth function which equals 1 on $[\epsilon^4, R_{\epsilon}]$ and is supported on $[\epsilon^4/2, 2R_{\epsilon}]$. Then, using Lemma \ref{lem:Cauchy}, we have
$$ \int_0^{\infty} \Psi(t) \log t d\nu_{nn^{\prime}}(t) = \int_0^{\infty} \Psi(t) \log t d\nu_{nn}(t) + O(\epsilon) $$
if $\delta$ is sufficiently small, depending on $\epsilon$ and $\Psi$.
By Corollary \ref{cor:Lweakly}, we have that $\mu_{\frac{1}{n}A_nA_n^{*}}$ and $\mu_{\frac{1}{n}B_{n}B_{n}^{*}}$ converge to the same limit, and thus $nu_{nn}$ converges to zero. Thus,
$$ \int_{\epsilon^4}^{R_\epsilon} \log t d\nu_{nn^{\prime}}(t) \leq \int_0^{\infty} \Psi(t) \log t d\nu_{nn^{\prime}}(t) = O(\epsilon) $$
\vspace{1pc}

(3) The region of moderately small $t$: \\
Consider the region $t \in [\delta^2, \epsilon^4]$. We wish to show that 
$$ \int_{\delta^2}^{\epsilon^4} |\log t||d\nu_{nn^{\prime}}(t)| = O(\epsilon)$$
By the triangle inequality and symmetry, it suffices to show that, with probability 1, one has 
$$ \int_{\delta^2}^{\epsilon^4} |\log t||d\mu_{\frac{1}{n^{\prime}} A_{nn^{\prime}}A_{nn^{\prime}}^{*}}(t)| = O(\epsilon)$$
for all but finitely many $n$. We can express the left hand side as
$$ \frac{1}{n} \sum_{i=1}^{n^{\prime}} f(\frac{1}{\sqrt{n}} \sigma_i(A_{nn^{\prime}})) $$

where $f(t) = |\log t| \mbb{1}_{(\delta^2 \leq t^2 \leq \epsilon^4)}$.  \\

As $f$ is less than $|\log\delta|$, if $\delta < \epsilon^2$, we may make the contribution for $i \geq (1 - 2\delta)n$ acceptable. Thus, it suffices to show that we have, almost surely,
$$ \frac{1}{n} \sum_{1 \leq i \leq (1-2\delta)n} f(\frac{1}{\sqrt{n}} \sigma_i(A_{nn^{\prime}})) = O(\epsilon)$$
for all but finitely many $n$. \\

Recall that $n^{\prime} = \lfloor (1-\delta)n \rfloor$. For any $0<c<1$, by Proposition \ref{prop:tail} and the Borel-Cantelli lemma, we have with probability 1 that
$$ \dist(Z_i, span(Z_1, ..., Z_{i-1}, Z_{i+1}, ... , Z_{n^{\prime}})) \geq c\sqrt{n-n^{\prime}} = c\sqrt{\delta n} $$ 
for all but finitely many $n$ and all $1 \leq i \leq n^{\prime}$. Thus,
$ \frac{1}{\sqrt{n}} \sigma_i(A_{n, n^{\prime}}) \geq c\sqrt{\delta} $, so that 
$$ \frac{1}{n} \sum_{i=1}^{n^{\prime}} (\frac{1}{\sqrt{n}} \sigma_i(A_{nn^{\prime}})) ^{-2} \leq \frac{n^{\prime}}{n} \frac{c}{\delta} = O_{\delta}(1)  $$ 
Moreover, as $n^{\prime} = \lfloor (1-\delta)n \rfloor$ and $\sigma_i(A_{n, n^{\prime}})$ is decreasing in $i$, one has that $\frac{1}{\sqrt{n}} \sigma_{\lfloor (1-2\delta)n \rfloor} (A_{n, n^{\prime}}) \geq c\sqrt{\delta} $.\\

Now, using Proposition \ref{prop:tail} and the Borel-Cantelli lemma again, we have with probability one, 
$$ \dist(Z_i, span(Z_1, ..., Z_{i-1}, Z_{i+1}, ... , Z_{n^{\prime\prime}})) \geq c\sqrt{n-n^{\prime\prime}} $$ for all but finitely many $n$, all $1 \leq i \leq n^{\prime\prime}$ and $\frac{n}{2} \leq n^{\prime\prime} \leq n^{\prime}$. Thus, 
$$ (\frac{1}{\sqrt{n}} \sigma_i(A_{n, n^{\prime\prime}}) )^{-2} \leq \frac{cn}{n-n^{\prime\prime}} $$ so that we have almost surely that
$$ \frac{1}{n} \sum_{i=1}^{n^{\prime\prime}} (\frac{1}{\sqrt{n}} \sigma_i(A_{n, n^{\prime\prime}}) )^{-2} = O(\frac{n}{n-n^{\prime\prime}}) $$ for all but finitely many $n$ and all $\frac{n}{2} \leq n^{\prime\prime} \leq n^{\prime}$.

Using the last $n - n^{\prime\prime}$ terms in the sum on the left hand side below, we get that 
$$  \sum_{i=1}^{n^{\prime\prime}} (\frac{1}{\sqrt{n}} \sigma_i(A_{nn^{\prime\prime}})) ^{-2}  \geq (n-n^{\prime\prime}) (\frac{1}{\sqrt{n}} \sigma_{2n^{\prime\prime} - n}(A_{nn^{\prime\prime}})) ^{-2} $$ we can conclude that 
$$ (\frac{1}{\sqrt{n}} \sigma_{2n^{\prime\prime} - n}(A_{nn^{\prime\prime}})) \geq c^{\prime} \frac{n - n^{\prime\prime}}{n}$$
for all but finitely many $n$ and $n/2 \leq n^{\prime\prime} \leq n^{\prime}$, for some constant $c^{\prime} > 0$. Using Lemma 4.7, we can conclude that 
\begin{equation}  \frac{1}{\sqrt{n}} \sigma_{i}(A_{nn^{\prime}}) \geq c \frac{n^{\prime} - i}{n} \tag{*}  \end{equation} 
for all but finitely many $n$ and all $1 \leq i \leq (1-2\delta)n$. \\

Now recall, 
$$ \frac{1}{n} \sum_{1 \leq i \leq (1-\delta)n} f(\frac{1}{\sqrt{n}} \sigma_i(A_{nn^{\prime}}))$$
By (*), we see that the only terms in this sum that do not vanish are those with for which $i = (1-O(\epsilon^2))n$. For such terms, using (*) and the fact $f(t) \leq -\log t$, we have that
$$ \frac{1}{n} \sum_{1 \leq i \leq (1-2\delta)n} f(\frac{1}{\sqrt{n}} \sigma_i(A_{nn^{\prime}})) = O(\epsilon)$$

\vspace{1pc}

(4) The region of small $t$:\\
Consider $t \leq \delta$. As in the region of moderately small $t$, we need only show that 
$$ \frac{1}{n} \sum_{i=1}^{n^{\prime}} g(\frac{1}{\sqrt{n}} \sigma_i(A_{nn^{\prime}})) = O(\epsilon)$$
for all but finitely many $n$, where $g(t) = |\log t| \mbb{1}_{(t^2 < \delta^2)}$. 

By Proposition \ref{prop:tail} and the Borel-Cantelli lemma, we that with probability 1, for large enough $n$ and all $i \leq n^{\prime}$,
$$ \dist(Z_i, span(Z_1, ..., Z_i, Z_{i+1}, ... , Z_{n^{\prime}})) \geq c\sqrt{\delta }. $$
Again, using Proposition \ref{prop:tail} and Lemma \ref{lem:dist}, we have that 
$$ \frac{1}{n^{\prime}} \sum_{i=1}^{n^{\prime}} ( (\frac{1}{\sqrt{n}} \sigma_i(A_{nn^{\prime}}))^{-2} = O(\epsilon).$$
If $\delta$ is small enough, we have $g(t) \leq \epsilon/t^2$, and the result follows.

\end{proof}

\vspace{5pc}

\newcommand{\etalchar}[1]{$^{#1}$}
\providecommand{\bysame}{\leavevmode\hbox to3em{\hrulefill}\thinspace}
\providecommand{\MR}{\relax\ifhmode\unskip\space\fi MR }
\providecommand{\MRhref}[2]{%
  \href{http://www.ams.org/mathscinet-getitem?mr=#1}{#2}
}
\providecommand{\href}[2]{#2}

\end{document}